\documentclass[12pt]{article}
\usepackage{amsmath,amssymb,amsthm}
\usepackage[margin=1in]{geometry}
\usepackage{hyperref}
\usepackage{dsfont} 
\usepackage{enumerate}
\usepackage{amsthm,amsmath,amssymb}
\usepackage{tikz}
\usepackage{caption}
\usepackage{changepage}

\usepackage{graphicx}
\usepackage{color}

\usepackage[algo2e,ruled,vlined]{algorithm2e}
\usepackage[normalem]{ulem}
\usepackage{algpseudocode}
\usepackage{color}
\usepackage{ mathrsfs }
\usepackage{mathtools}
\usepackage{rotating}
\usepackage[utf8]{inputenc}

\newcommand{\RR}{\mathbb{R}}

\newtheorem{thm}{Theorem}[section]
\newtheorem{lem}[thm]{Lemma}
\newtheorem{prop}[thm]{Proposition}

\newtheorem{cor}[thm]{Corollary}

\newtheorem{rem}[thm]{Remark}
\theoremstyle{definition}
\newtheorem{defn}[thm]{Definition}
\newtheorem{ex}[thm]{Example}

\newcommand{\mult}{{\rm mult}}

\begin{document}

\title{Strictly Interlaced Spectral Data for the Weighted Matching Polynomial of a Graph}

\author{
Shaun Fallat\thanks{Department of Mathematics and Statistics, University of Regina,  Regina, Saskatchewan, S4S0A2, Canada. (shaun.fallat@uregina.ca).}
\and
Johnna Parenteau \thanks{Department of Mathematics and Statistics, University of Regina,  Regina, Saskatchewan, S4S0A2, Canada. (parenjoh@uregina.ca).} 
}


\maketitle

\begin{abstract} 
Interlacing of the real roots of a weighted matching polynomial for a graph $G$ and that of a vertex-deleted subgraph is classical and well-known. In the context of strict interlacing of distinct roots, a demonstrated graph construction gives rise to a new classification of graphs, called ${ \rm SRSI}$ graphs, which include graphs that contain a Hamilton path. Graphs with a perfect (or nearly perfect) matching are shown to exhibit the SRSI$_w(v)$ property with respect to a particular weighting and for specific vertices, and all such graphs are characterized via this graph construction. As a consequence, we also characterize the trees that possess the SRSI property for all edge weightings and all vertices. 
\end{abstract}




\section{Introduction}

For a simple graph, $G=(V,E)$, a \emph{matching} $M$ in $G$, is a set of pairwise disjoint edges, and any vertex incident with a matched edge is said to be {\em saturated}; otherwise the vertex is {\em unsaturated}. We say $M$ is a $k$-matching if $M$ has size $k$.  A \emph{maximal matching} is a matching $M$, in which no additional edge can be added whilst maintaining the disjoint edge property. A matching is said to be \emph{maximum} if it contains the largest number of edges possible. If a matching $M$ saturates all vertices in the graph it is called a {\em perfect matching}, and if $M$ saturates all but one vertex in a graph, we call $M$ a {\em nearly perfect matching}.  A perfect matching $M$ has size $|M|=|V|/2$, assuming that $|V|$ is even, and a nearly perfect matching $M$ has size $|M|=(|V|-1)/2$ if $|V|$ is odd. 

Matchings have a rich history in graph theory partly because of their vast applications in network and scheduling problems, theoretical chemistry, and even statistical physics. Beginning with K\"onig's Theorem, the study of matchings arose in the early 1900s and has since grown with continued interest. Initial contributions regarding the theory of matching polynomials, however, came from Heilmann and Lieb as a means to study the monomer-dimer problem in statistical physics (\cite{hl1}, \cite{hl2}), or equivalently, perfect or nearly-perfect matchings, but, interestingly, they failed to name their polynomial. Coined as the reference, acyclic, or king polynomial, depending on the field of interest, the matching polynomial has defined in several ways over time, but they all align with the following notation 
$$ m(G,x) = \underset{k}{\sum}(-1)^k\mu(G,k)x^{n-2k}, \label{mpolydef}$$
where $\mu(G,k)$ denotes the number of $k$-matchings in a graph, $G$. We refer the reader to the works \cite{bf}, \cite{g}, \cite{gg}, \cite{hl1}, and \cite{hl2} for a comprehensive overview of the existing literature surrounding the matching polynomial; however, the work presented here will not rely on these existing results unless explicitly stated. As a means to extend the traditional matching polynomial $m(G,x)$ from simple unweighted graphs to all simple weighted graphs, we alter the underlying assumption of $m(G,x)$ to accommodate weighted graphs. Of particular interest here is an investigation into the roots of this weighted matching polynomial, and, in particular, weighted graphs whose 
weighted matching polynomial has real and distinct roots. 

The remainder of this paper is divided into three sections. The next section outlines essential terminology and notation and presents a brief survey of foundation results concerning weighted matching polynomials. Section 3 contains the main results of this work including a characterization of all weighted graphs that possess a strict interlacing property between the weighted matching polynomial of a graph and a vertex-deleted subgraph.

\section{Preliminaries} 
Given a simple graph $G=(V,E)$, we consider the weightings of the edges in $E$, or simply a function $w : E \rightarrow \RR^+$ and denote the weight of an edge $e=\{a,b\}$ as $w(e)$ or $w_{ab}$ as dictated by context. For any graph $G=(V,E)$, suppose that $X\subset V$ and $Y \subset E$. Then, we let $G \setminus X$ (resp. $G \setminus Y$) denote the subgraph obtained from $G$ by removing the vertices in $X$  (resp. by removing the edges in $Y$).

Suppose $M$ is a \emph{k}-matching in a weighted graph, $G$. Then the \emph{weight of M} \label{W(M)def}is 
$$ w(M) = \underset{e_j \in M}{ \prod} w(e_j).$$
By adding the weights of all possible $k$-matchings in $G$, for any value of $k$, we define the \emph{weighted k-matching number} \label{mudef} as 
$$ \mu_w(G,k) = \underset{\text{$k$-matchings in $G$}}{ \hspace{-0.75cm}\sum } \underset{\hspace{0.5cm} e_j \in M}{ \hspace{0.5cm}\prod} w(e_j).$$
For our purposes, the \emph{weighted matching polynomial}  is defined as
 
 $$ m_w(G,x) = \overset{\lfloor \frac{n}{2} \rfloor}{\underset{k = 0}{\sum}}(-1)^k\mu_w(G,k)x^{n-2k}. \label{mwpolydef}$$ 
 
Since $\mu_w(G,0) = 1$, it follows that the weighted matching polynomial is a monic polynomial. Observe that for any graph, $G$, the \emph{k-matching number} coincides with $\mu(G,k)$ when $w(e_i)=1$ for all edges, $e_i \in E$. When a graph is unweighted, its weight function will be denoted as $w \equiv 1$. 
Several versions of a weighted matching polynomial have been studied and analyzed; see, for example, \cite{bf, gbook, g, gg, hl1, hl2, kuchen}. 

A weighted matching polynomial has several interesting and well-known properties and associated expressions, particularly when a vertex or edge is removed from a (weighted) graph, leading to various {\em bridge formulations}, versions of which can be found in \cite{gbook}.

\begin{prop}
    Suppose $G$ is a given weighted graph on $n$ vertices with specified edge weight function $w$.
    Then:
    \begin{enumerate}
\item If $G$ is disconnected with components
$G_1, G_2, \ldots, G_s,$ where $G_i =(V_i, E_i, w)$ for all $i$, then $m_{w}(G,x) = m_w(G_1,x)\cdot m_w(G_2, x) \cdots m_w(G_s,x);$ 
\item  For any edge, $e = \{ i,j\}$ in $E$ it follows $$m_{w}(G,x) = m_w({G \setminus\{e\}},x) - w(e) \cdot m_{w}(G\setminus\{i,j\}, x);$$
\item For any vertex, $v$, in $V$, $$m_{w}(G,x) = x \cdot m_w({G\setminus \{v\}},x) - \underset{v_i\sim v}{\sum} w_{vv_i} m_{w}(G\setminus\{v,v_i\}, x),$$ where $v_1, v_2, \ldots, v_s$ are the neighbours of $v$.
\end{enumerate}
\label{bridges}
\end{prop}

For completeness, we sketch a verification of item (3) in Proposition \ref{bridges}, which is referred to several times in the next section. 
Suppose $G = (V,E,w)$ is a graph with fixed weight function, $w$, and let $v \in V$. For every neighbour, $v_1, v_2, \ldots, v_s$, of $v$, let $e_i$ be the edge incident with $v$ and $v_i$. Applying item (2) in Proposition \ref{bridges} with respect to the edges, $e_i$, 
\begin{align*}
m_w(G,x) &= \underbrace{m_w({G \setminus \{e_1\}},x)} - w(e_1) \cdot m_{w}(G\setminus\{v,v_1\}, x)\\
&\hspace{1.5cm} \downarrow \text{ apply (2) } \\
&= \underbrace{m_w({G \setminus \{e_1,e_2\}},x)} - w(e_2) \cdot m_{w}(G\setminus\{v,v_2\}, x) - w(e_1) \cdot m_{w}(G\setminus\{v,v_1\}, x) \\
& \hspace{1.5cm} \downarrow \text{continue to apply (2) as needed} \\
&= m_w({G \setminus \{e_1,e_2, \ldots, e_s\}},x)- \underset{v \sim v_i}{\sum}w_{vv_i}m_w(G \setminus \{v,v_i\},x). 
\end{align*}
Observe that $G \setminus \{e_1,e_2, \ldots, e_s\} = (G \setminus \{v\}) \cup \{v\}$ and therefore applying (1) from the proposition yields the desired relation.

As in the well-studied unweighted case, the weighted matching polynomial for a tree can be reinterpreted as the characteristic polynomial of a certain matrix associated with a tree. Suppose $G$ is a graph on $n$ vertices with specified weight function $w$. For $G$, we associate a real symmetric $n \times n$ matrix $A=[a_{ij}]$ (and write $A \in  S^o_+(G)$) by setting $a_{ij} = w_{ij}$ whenever $i \neq j$ and $\{i,j\} \in E$; otherwise, $a_{ij}=0$. In particular, all main diagonal entries of $A$ are zero, which is known as a hollow matrix (see \cite{hollow}). In the case of a tree $T$, if $A \in  S^o_+(T)$, it follows that $\det A[S]$ (here $A[S]$ means the principal submatrix of $A$ whose rows and columns are indexed by $S$) is non-zero if and only if the induced subgraph $T[S]$ contains a (maximum) matching of size $|S|/2$. In particular, $|S|$ must be even in this case.
Furthermore, if $M$ is a maximum matching of size $|S|/2$ in the induced subgraph $T[S]$, then
\[ \det A[S] = (-1)^{\frac{|S|}{2}} \underset{e_j \in M}{\prod} w(e_j). \]

From the above analysis, it is straightforward and well-known to prove the following identity involving such matrices associated with trees, which is easily extended to forests.

\begin{prop} 
For a tree, $T$, with weighted edges, $\{e_1, \ldots, e_{n-1} \}$, and $A \in S^o_+(T)$ such that $a_{ii} = 0$ and $a_{ij}=\sqrt{w(e)}>0$  if $e=\{ i,j\}$ is an edge in $T$. Then 
$$\det(xI-A) = m_w(T,x).$$
\label{mw-tree}
\end{prop} 

\vspace*{-0.75cm}

Studies involving properties of the roots of the (weighted) matching polynomial are fundamental in the arena of algebraic matching theory. For any weighted graph $G$, we let $\rho(m_w(G, x))$ denote the collection of roots of $m_w(G,x)$. Furthermore, if $\lambda \in \rho(m_w(G,x))$, then let $\mult(m_w(G,x);\lambda)$ denote the multiplicity of $\lambda$ as a zero of $m_w(G,x)$. To conclude, we state two foundational results concerning the roots of the weighted matching polynomial (see also \cite{gbook}).

\begin{thm} 
Let $G=(V,E,w)$ be a graph with fixed weight function, $w$. Then, the roots of $m_w(G,x)$ are real and are symmetric about 0. 
\label{realroots}
\end{thm}

\begin{thm}
Let $G=(V,E,w)$ be a graph with a fixed weight function, $w$. Then, the roots of $m_w(G\setminus \{v\},x)$ interlace the roots of $m_w(G,x)$.  
\label{interlace}
\end{thm}

\section{Simple Roots and Strict Interlacing in the Weighted Matching Polynomial}

 It has long been known that the eigenvalues of a symmetric, irreducible, and tridiagonal matrix are real and simple (that is, distinct). In other words, any symmetric matrix whose off-diagonal matrix pattern fits a path requires simple and distinct eigenvalues. Additionally, eigenvalue interlacing is an important and well-established property associated with symmetric matrices. To this end, 
if $ \lambda_1 \leq \lambda_2 \leq \cdots \leq \lambda_n$ are the eigenvalues of an $n \times n$ symmetric matrix $A$ and 
$ \mu_1 \leq \mu_2 \leq \cdots \leq \mu_{n-1}$ are the eigenvalues of the $(n-1) \times (n-1)$ principal submatrix of $A$ obtained by deleting row and column $i$, denoted by $A(i)$ ($1 \leq i \leq n$), then it follows that the eigenvalues of $A$ interlace the eigenvalues of $A(i)$ given by
\[ \lambda_1 \leq \mu_1 \leq \lambda_2 \leq \mu_2 \leq \cdots \leq \lambda_{n-1} \leq \mu_{n-1} \leq \lambda_n.\] In the case of 
$n \times n$ symmetric, irreducible, and tridiagonal matrices, more is known along these lines, namely,
\begin{equation} \lambda_1 < \mu_1 <\lambda_2 < \mu_2 < \cdots < \lambda_{n-1} < \mu_{n-1} < \lambda_n\label{strict-il} \end{equation}
whenever $i=1$ or $i=n$. We refer to the strict inequalities in (\ref{strict-il}) as {\em strict-interlacing inequalities}. In summation, the eigenvalues of any $n \times n$
symmetric, irreducible, and tridiagonal matrix $A$ are real, simple, and satisfy strict interlacing with respect to $A(1)$ or $A(n)$.

The main focus of this section is to study weighted graphs $G$ with the property that the roots of $m_w(G,x)$ are simple and, in addition, there exists a vertex $v$ in $G$ such that the roots of $m_w(G,x)$ and $m_w(G \setminus \{v\},x)$ strictly interlace. To this end, we make the following definition.

\begin{defn} \label{srsi}
A graph, $G =(V,E,w)$, on $n \geq 2$ vertices is said to be \emph{{\rm SRSI} with respect to vertex $v$ and a fixed weight function, $w$}, denoted ${\rm SRSI}_w(v)$, if: 
\begin{itemize} 
	\item[1)]{ $m_w(G,x)$ has simple roots $({\rm SR}_w)$, and } 
	\item[2)]{ $m_w(G\setminus \{v\},x)$ has simple roots that strictly interlace $m_w(G,x)$.} 
\end{itemize} 

If $G$ is SRSI$_w(v)$, then we call $v$ a ${\rm SRSI}_w$ vertex for $G$. Furthermore, the following acronyms will be used in relation to the above two properties.

\medskip 
\medskip 

\begin{tabular}{ |p{3cm}||p{12cm}| }
 \hline
 \multicolumn{2}{|c|}{Notation for Simple Roots and Strict Interlacing} \\
 \hline
 \textbf{Acronym} & \textbf{Definition}\\
 \hline
SR   & $G$ has simple roots for all weight functions\\
SR$_w$ & $G$ has simple roots with respect to the weight function, $w$ \\ \hline 
SRSI$_w(v)$ & $G$ has simple roots that strictly interlace when $v$ is removed and with respect to the weight function, $w$ \\
SRSI$(v)$ & $G$ has simple roots that strictly interlace when $v$ is removed and with respect to all weight functions  \\
SRSI$_w$ & $G$ has simple roots that strictly interlace when any vertex is removed and with respect to the weight function, $w$ \\
SRSI & $G$ has simple roots that strictly interlace when any vertex is removed and with respect to all weight functions\\
 \hline
\end{tabular}
\end{defn}

The graph below is an example of a graph (without a Hamilton path) that is ${ \rm SRSI}_w(v)$ for some of its vertices but for all edge weightings. 

\begin{ex}
\begin{figure}[!htb]
\centering
\includegraphics[scale=.85]{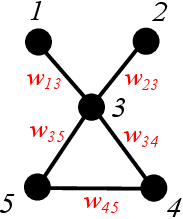} 
\caption{A graph which is ${ \rm SRSI}(1)$.} 
\label{WAntenna}
\end{figure} 

Consider the graph $G = (V,E,w)$ with a fixed weight function $w$ in Figure \ref{WAntenna}. Using Proposition \ref{bridges}, 
\begin{align*}
m_{w}(G,x) = x\cdot m_w({G\setminus \{1\}},x) - w_{13}\cdot  m_{w}(G\setminus\{1, 3\}, x).
\end{align*} 
Assume the roots for $m_{w}(G, x)$ are $\{ \lambda_1, \lambda_2, \ldots, \lambda_5 \}$, and denote the roots for $m_w(G \setminus \{1\}, x)$ by  $\{ \mu_1, \mu_2, \mu_3, \mu_4\}$. Observe $G\setminus\{1, 3\} = K_1 \cup K_2$, so $\rho \big(m_{w}(G\setminus\{1, 3\}, x)\big) = \{0, \pm \sqrt{w_{45}} \}$. Moreover, 
$m_w(G \setminus \{1\}, x) = x^4 - (w_{23}+w_{34}+w_{35}+w_{45})x^2 + w_{23}w_{45}$, 
and since $G \setminus \{1\}$ has a perfect matching, each $\mu_i \neq 0$ for $1 \leq i \leq 4$. Evaluating $m_w(G \setminus \{1\},x)$ at the root $\sqrt{w_{45}},$
\begin{align*}
m_w(G \setminus \{1\}, \sqrt{w_{45}}) &= {w_{45}}^2 - (w_{23}+w_{34}+w_{35}+w_{45}){w_{45}} + w_{23}w_{45} \\
&= -w_{45}(w_{34} + w_{35}) \\
&\neq 0. 
\end{align*} 
Thus, $ \pm \sqrt{w_{45}}$ is not a root of $m_w(G \setminus \{1\},x)$ or $m_w(G, x)$. In order to show the roots of $m_w(G \setminus \{1\}, x)$ strictly interlace the roots of $m_w(G, x)$, we further evaluate the bridge formula at the root $\mu_i$.  The first term on the right-hand side of the bridge formula vanishes, and the sign of the remaining terms  gives a negative. 
$$m_{w}(G,\mu_{i}) = \mu_i\cdot \underbrace{m_w(G \setminus \{1\}, \mu_{i})}_\text{ $ =0$}- ~\framebox[1.0\width]{$~ w_{13} m_{w}(G \setminus \{1, 3\}, \mu_{i})$.~}$$ It easily follows that $m_{w}(G, x)$ has a real root in each interval $(\mu_1, \mu_2), (\mu_2, \mu_3),$ and $(\mu_{3}, \mu_{4})$, which gives rise to the simple roots, $\lambda_2, \lambda_3,$ and $ \lambda_4$. Thus, we may conclude that the roots of $m_{w}(G, x)$ strictly interlace the roots of $m_{w}(G, x)$. Hence, $G$ is ${\rm SRSI}(1)$.
\end{ex}

The remainder of this section is divided into three subsections. First, we define a specialized graph operation (called expansion) and use it to construct graphs that are SRSI$_w(v)$ for a particular vertex $v$. Second, by extending a known result of L. Duarte on interlaced spectral data for trees, we provide a characterization of all SRSI$_w(v)$ graphs using this graph expansion operation. Finally, we provide a characterization of all SRSI trees.

\subsection{Characterizing all SRSI$_w(v)$ Graphs}

It is evident that any graph $G=(V,E)$ for which all roots of $m_w(G,x)$ are simple must possess a perfect matching (if $|V|$ is even) or a nearly perfect matching (if $|V|$ is odd). In this section, we verify that, in fact, the converse also holds. Namely, for any graph with a perfect or a nearly perfect matching, there exists an edge weighting $w$ and a vertex $v \in V$ such that $G$ is SRSI$_w(v)$. Thus, a primary focus in this subsection is graphs with perfect or a nearly perfect matchings. 

In \cite{MN}, it is shown that for any tree $T$ with a perfect matching, there exists a pendant vertex $v$ incident to a vertex $u$ of degree two, such that the tree $T\setminus \{u,v\}$ also has a perfect matching. In fact, more can be said about trees with perfect or nearly perfect matchings, which we organize into the following remark.

\begin{rem} $\,$ \\
\label{match-rem}
\vspace*{-0.5cm}
\begin{enumerate}
    \item Suppose $T$ admits a perfect matching. Then for any vertex $v$ of $T$ let $T \setminus \{v\}$ have components $C_1, \ldots, C_k$. From the proof of Lemma 3 it follows that all but one of the components  $C_1, \ldots, C_k$ have a perfect matching while the remaining component possess a nearly perfect matching.
    \item Suppose $T$ admits a nearly perfect matching. Then using Lemma 3, we know there exists an unsaturated pendant vertex, and the sub-tree obtained by removing this pendant vertex has a perfect matching. In fact, let $v$ be any unsaturated vertex associated with a nearly perfect matching of $T$. Then assume that $C_1, \ldots, C_k$ are the components of $T \setminus \{v\}$. It is easy to check that all the components $C_1, \ldots, C_k$ possess perfect matching.
\end{enumerate}
\end{rem}

\subsection{Duarte Property: Extensions to the Hollow Case}
In 1989 Duarte \cite{Leal} studied a particular inverse eigenvalue problem for trees. To state this result, we need some additional notation. To any graph $G=(V,E)$ on $n$ vertices, we can associate a class of symmetric matrices, $S(G)$ consisting of all $n \times n$ symmetric matrices $A=[a_{ij}]$ where $a_{ij} \neq 0$ whenever $i\neq j$ and $i$ and $j$ are adjacent in $G$. Note that $S^o_+(G)$ is a subset of $S(G)$. To any $n \times n$ matrix $A$, we let $\sigma(A)$ denote the eigenvalues of $A$, and let $A(i)$ denote the $(n-1) \times (n-1)$ principal submatrix obtained from $A$ by deleting the row and column $i$. The following result was proved in \cite{Leal}

\begin{thm}\cite{Leal}
\label{leal} Suppose that two strictly increasing sequences of real numbers $\stackrel{\rightarrow}{\lambda}$ of length $n$ and the other $\stackrel{\rightarrow}{\mu}$ of length $n-1$ are given, such that   $\stackrel{\rightarrow}{\mu}$ strictly interlaces $\stackrel{\rightarrow}{\lambda}$. Then, for any tree $T$, there exists a matrix $A \in S(T)$, $v \in T$, such that $\sigma(A) = \stackrel{\rightarrow}{\lambda}$ and 
$\sigma(A(v)) = \stackrel{\rightarrow}{\mu}$. 
\end{thm}




We wish to extend Theorem \ref{leal}  to hollow matrices. Recall that a matrix $A=[a_{ij}]$ is called {\em hollow} if $a_{ii}=0$ for all $i=1,2,\ldots, n$. An increasing sequence $\stackrel{\rightarrow}{\lambda}$ of real numbers is {\em symmetric about} 0 if $\lambda_{i} = -\lambda_{n-i+1}$ for all $i=1,2,\ldots, n$. We now prove the following Duarte-type theorem concerning an inverse eigenvalue problem for the case where hollow matrices admit distinct spectra.

\begin{thm}
Suppose $T$ is a tree on $n$ vertices that possesses a perfect or nearly perfect matching. Let $\stackrel{\rightarrow}{\lambda}$ be a real sequence of length $n$ and $\stackrel{\rightarrow}{\mu}$ a real sequence of length $n-1$, both consisting of distinct real numbers symmetric about 0 such that  $\stackrel{\rightarrow}{\mu}$ strictly interlaces $\stackrel{\rightarrow}{\lambda}$. Then there exists a hollow matrix $A \in S(T)$ and a vertex $v \in T$ such that  $\sigma(A) = \stackrel{\rightarrow}{\lambda}$ and 
$\sigma(A(v)) = \stackrel{\rightarrow}{\mu}$.
In fact, $A$ can be chosen in $S^o_+(G)$.
\label{H-Leal}
\end{thm}

\begin{proof}
    The proof uses induction on the number of vertices, $n$, in the tree $T$. For $n=2$ or $3$, the tree $T$ is simply a path, and from \cite{AG} this result follows for Jacobi matrices. Assume that the result holds for all such trees on fewer than $n$ vertices, and let $T$ be a tree as assumed in the theorem above.

Let $\stackrel{\rightarrow}{\lambda} =(\lambda_1, \lambda_2, \ldots, \lambda_n)$  and let $\stackrel{\rightarrow}{\mu}=(\mu_1, \mu_2, \ldots, \mu_{n-1})$. In addition, we denote by
\[ f(x)= \prod_{i=1}^n (x-\lambda_i), ~~~ g(x)= \prod_{j=1}^{n-1} (x-\mu_i),\] polynomials of degrees $n$ and $n-1$, respectively. 
According to the proof of the main theorem in \cite{Leal}, we can write
\[ \frac{f(x)}{g(x)} = (x-a) - \sum_{k=1}^{m} y_k \frac{h_k(x)}{g_k(x)},\] where $a,y_1, \ldots y_m$ are unique real numbers with $y_k>0$ for all $k$, and $h_1, \ldots h_m$ are unique monic polynomials with $\deg(h_k) < \deg(g_k)$  and $g(x) = \prod g_k(x)$. It is not difficult to verify that, via an application of the division algorithm, if $f$ is written as $f(x)=xg(x)+h(x)$. Then the roots of the remainder, $h(x)$ are also symmetric about 0. Hence, it follows that $a=0$. Furthermore, in \cite{Leal} it was shown that the roots of $h_k$ and $g_k$ are strictly interlaced. Since $T$ is assumed to have a perfect or nearly perfect matching, applying Remark \ref{match-rem},  we know there exists a (cut) vertex $v$ in $T$ such that each component of $T \setminus \{v\}$ has a perfect or nearly perfect matching, with at most one component possessing a nearly perfect matching. Assume that the components (or branches) of $T \setminus \{v\}$ are denoted by $B_1, \ldots B_m$ and factor  $g(x) = \prod g_k(x)$, such that each $g_k$ has distinct roots symmetric about 0. Then by induction, for each $k=1,2,\ldots, m$, we can find matrices symmetric hollow matrices $A_k$ with characteristic polynomial $g_k$ and for which $h_k(x)$ is a polynomial corresponding to the branch $B_k$ with the neighbour of $v$ removed, so by construction $h_k$ will have roots symmetric about 0. Finally, following the description of the entries in row/column $v$, namely $\sqrt{y_k}$ and setting $a=0$, we will have constructed a symmetric hollow matrix $A$ in $S(T)$ with  characteristic polynomial $f(x)$ and where $A(v)$ has characteristic polynomial $g(x)$.
\end{proof}




Since our focus is on graphs whose weighted matching polynomial admit distinct roots, we must consider graphs with a perfect matching or a nearly perfect matching. The next result is almost certainly known and is not difficult to establish, but is needed for our main characterization of connected SRSI$_w(v)$ graphs. We note here that these results naturally extend to the disconnected case, but we omit the proof details here.

\begin{lem} \label{pm-npm}
    Suppose $G$ is a connected graph that has a perfect matching or a nearly prefect matching, respectively. Then there exists a spanning tree $T$ of $G$ that has the same perfect or the same nearly perfect matching, respectively.
\end{lem}

We are now in a position to characterize all graphs that have the SRSI$_w(v)$ property.

\begin{thm} \label{srsi-thm}
    Suppose $G$ is a connected graph that has a perfect matching or a nearly perfect matching, respectively. Then, there exists an edge weighting $w$ and a vertex $v$ such that $G$ is SRSI$_w(v)$. 
\end{thm}

\begin{proof}
Assume $G$ has a perfect matching. Then, by Lemma \ref{pm-npm}, there exists a spanning tree $T$ of $G$ that has the same perfect matching. Thus, using Theorem \ref{H-Leal}, there exists a hollow matrix $A$ in $S(T)$ such that 
  $\sigma(A) = \stackrel{\rightarrow}{\lambda}$ and 
$\sigma(A(v)) = \stackrel{\rightarrow}{\mu}$, where  $\stackrel{\rightarrow}{\lambda}$ is a real sequence of length $n$ and $\stackrel{\rightarrow}{\mu}$ is a real sequence of length $n-1$, both consisting of distinct real numbers symmetric about 0 such that  $\stackrel{\rightarrow}{\mu}$ strictly interlaces $\stackrel{\rightarrow}{\lambda}$. We may assume that $A \in S^o_+(T)$, since $T$ is a tree. By Proposition 
\ref{mw-tree}, it follows that there exists an edge weighting $w$ of $T$ such that the roots of $m_{w}(T,x)$ are $\stackrel{\rightarrow}{\lambda}$ and the roots of $m_{w}(T \setminus \{v\},x)$ are $\stackrel{\rightarrow}{\mu}$. Assign the weight of $\varepsilon>0$ to each edge in $E(G) \setminus E(H)$. Then for $\varepsilon >0$ sufficiently small, the roots of $m_w(G,x)$ are distinct and strictly interlace the distinct roots of $m_w(G\setminus \{v\},x)$. Hence $G$ is SRSI$_w(v)$. A similar argument, omitted here, can be applied when $G$ has a nearly perfect matching $M$, with $v$ being an unsaturated vertex with respect to $M$.
\end{proof}

\subsection{Constructing SRSI$_w(v)$ Graphs}

It turns out that we can use existing graphs that are ${\rm SRSI}_w(u)$ with respect to some vertex $u$ to create new graphs that are ${\rm SRSI}_w(v)$ where $v$ is a newly added neighbour of $u$. Vertices that are added in this way are necessarily adjacent to a ${\rm SRSI}_w$ vertex $u$ in the new graph, but can also have additional neighbours. This graph operation has the following definition and notation. 

\begin{defn}
Suppose $\widehat{G} = (\widehat{V},\widehat{E}, \widehat{w})$ is a graph. If $u \in \widehat{V}$, then $G = \widehat{G} \underset{u} {\Cup} \{v\}$ is the weighted graph defined by the vertex set, $V = \widehat{V} \cup \{v\}$, where $u \sim v$ and edge set, $E = \widehat{E} \cup \{\hat{e}_1, \hat{e}_2, \ldots, \hat{e}_s\}$, where $\hat{e}_1 =\{u,v\}$, $\hat{e}_i =\{v, v_i \}$ for some $v_i \neq u $ in $\widehat{V}$, and weight function, $w|_{\widehat{V}} =\widehat{w}$. The edges, $\hat{e}_i =\{v, v_i \}$, can have any arbitrary weights. This operation is called \emph{graph expansion}, \label{expansiondef} and $\widehat{G}$ is said to be expanded at $u$ by $v$ with respect to the weight function, $\widehat{w}$.  
\end{defn} 

\begin{figure}[!htb]
\centering
\includegraphics[scale=.9]{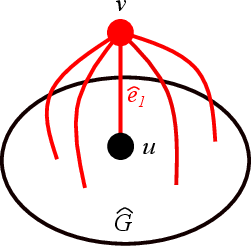} 
\caption{The expansion of $\widehat{G}$ at $u$ by $v$.} 
\label{GExpansion}
\end{figure} 

\begin{thm} 
Let $\widehat{G} = (\widehat{V},\widehat{E},\widehat{w})$ be a graph with a fixed weight function, $\widehat{w}$, which is ${ \rm SRSI}_{\widehat{w}}(u)$ for some vertex, $u$, in $\widehat{V}$. Then, the weighted graph, $G = \widehat{G} \underset{u}{\Cup} \{v\}$, is ${ \rm SRSI}_{w}(v)$.  
\label{SRSIv} 
\end{thm}

\begin{proof}
Suppose $\widehat{G} = (\widehat{V},\widehat{E},\widehat{w})$ is a graph with a fixed weight function, $\widehat{w}$, where $\widehat{G}$ is ${ \rm SRSI}_{\widehat{w}}(u)$ for some vertex, $u$. Consider the graph, $G = \widehat{G} \underset{u}{\Cup} \{v\}$. Using Proposition \ref{bridges}, 
\begin{align*}
 m_{w}(G,x) &= x \cdot m_{\widehat{w}}({G\setminus \{v\}},x) - \widehat{w}_{vu}m_{\widehat{w}}(G\setminus\{v, u\}, x) - \underset{ v_i \neq u}{\underset{v_i \sim v}{\sum}} \widehat{w}_{v_iv} m_{\widehat{w}}(G\setminus\{v, v_i\}, x). 
 \end{align*} 
 
Assume the collection of roots for $m_{w}(G,x)$ is $\{ \lambda_1, \lambda_2, \ldots, \lambda_n \}$. Similarly, we denote the collection of roots for $m_{\widehat{w}}(G \setminus \{v\},x)$ and $m_{\widehat{w}}(G \setminus \{v, u\},x)$ as $\{\mu_1, \mu_2, \ldots, \mu_{n-1} \}$ and $\{ \delta_1, \delta_2, \ldots, \delta_{n-2} \}$, respectively. Since $\widehat{G}$ is ${\rm SRSI}_{\widehat{w}}(u)$, the roots of $m_{\widehat{w}}(G \setminus \{v\},x)$ and $m_{\widehat{w}}(G \setminus \{v, u\},x)$ are simple, and the roots of $m_{\widehat{w}}(G \setminus \{v, u\},x)$ strictly interlace the roots of $m_{\widehat{w}}(G \setminus \{v\},x)$. Moreover, the roots of $m_{\widehat{w}}(G \setminus \{v, v_i\},x)$ interlace the roots of $m_{\widehat{w}}(G \setminus \{v\}, x)$ for all $v_i \in N_{G}(v) \setminus \{u\}$. When evaluating $m_{w}(G, x)$ at the root, $x = \mu_i$, $m_{\widehat{w}}(G \setminus \{ v\}, \mu_i) = 0$ and $m_{\widehat{w}}(G \setminus \{ v, u\}, \mu_i) = z$ where $z$ is a nonzero constant, and for every other neighbour, $v_i$, of $v$, we have $m_{\widehat{w}}(G \setminus \{ v, v_i\}, \mu_i) = \hat{z}$ where $\hat{z}$ is a non-positive or nonnegative constant which, when nonzero, has the same sign as $z$ by interlacing. It follows from the Intermediate Value Theorem, that  $m_{w}(G,x)$ has a real root in each interval $(\mu_1, \mu_2), (\mu_2, \mu_3),$ and $(\mu_{n-2}, \mu_{n-1})$, which gives rise to the simple roots, $\lambda_2, \lambda_3, \ldots, \lambda_{n-1}$. Applying basic limit arguments we conclude that all roots of $m_{w}(G,x)$ are real, and the roots of $m_{\widehat{w}}({G\setminus \{v\}},x)$ strictly interlace the roots of $m_{w}(G,x)$.
\end{proof}

It is known  that if $G=(V,E,w)$ is a graph with a Hamilton path, $P = v_1 v_2 \cdots v_n$, and any fixed weight function, $w$, the roots of $m_w(G, x)$ are simple, and the roots of $m_w(G \setminus \{v_1\}, x)$ (or $m_w(G \setminus \{v_n\}, x)$) strictly interlace the roots of $m_w(G, x)$. 
As a basic consequence, we can reach the same conclusion using the graph expansion operation defined above.

\begin{cor}
If $G$ is a graph with a Hamilton path, $P = v_1 v_2 \cdots v_n$, then $G$ is ${ \rm SRSI}(v_1)$ (respectively, ${ \rm SRSI}(v_n)$) for any weight function, $w$. 
\label{HamPath}
\end{cor} 

Furthermore, it easily follows that if $G$ possesses a Hamilton cycle, then $G$ is a SRSI graph. In fact, this idea can be pushed further, for example, with the Petersen graph. While the Petersen graph does not having a Hamilton cycle, it does admit a Hamilton path initiated at any vertex. Thus, the Petersen graph is also a SRSI graph.

Naturally, if vertices can be added to SRSI$_w(u)$ graphs in a certain manner to create more SRSI$_w(v)$ graphs, we question whether we can remove vertices in a SRSI$_w(v)$ graph and always remain SRSI$_w(u)$ where $u$ is a neighbour of $v$. However, this unfortunately turns out to be false. The edge weighting $w$ plays a key role in determining whether such SRSI$_w(u)$ neighbours exist. To see this, consider the subgraph $G\setminus \{v\}$, which is isomorphic to $P_7$ in Figure \ref{P7Iso}. 

\begin{figure}[!htb]
\centering
\includegraphics[scale=.75]{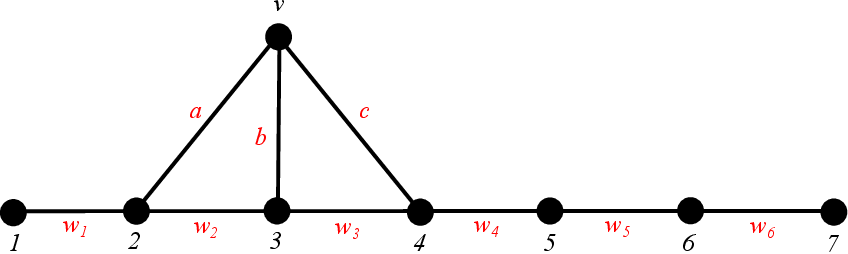} 
\caption{The converse of Theorem \ref{SRSIv} is false.} 
\label{P7Iso}
\end{figure} 

It follows that the matching polynomial associated with $G\setminus \{v\}$ is equal to
\begin{eqnarray*} m_w(G\setminus \{v\}, x) & = & x^7 -(w_1+w_2+w_3+w_4+w_5+w_6)x^5 \\
&&+\left(w_1(w_3+w_4+w_5+w_6)+w_2(w_4+w_5+w_6)+w_3(w_5+w_6)+w_4w_6\right)x^3 \\ &&-\left(w_1w_3(w_5+w_6) + w_1w_4w_6 + w_2w_4w_6\right)x. 
\end{eqnarray*}
We note here that since $G\setminus \{v\}$ has a Hamilton path, it is an SR graph for which  ~$m_w(G\setminus~\{v\}, x)$ always has zero as a root. Furthermore, it is not difficult to determine that for the graph, $G$, we have 
\begin{eqnarray*} m_w(G, x) & = & 
x\cdot m_w(G\setminus \{v\},x) - (a+b+c)x^6 +
\\ &&
\left(a(w_3+w_4+w_5+w_6)+
b(w_1+w_4+w_5+w_6) + c(w_1+w_2+w_5+w_6)\right)x^4 - \\ &&\hspace*{-2cm}\left(aw_3(w_5+w_6) + 
bw_1(w_4+w_5+w_6) + cw_1(w_5+w_6) +
cw_2(w_5+w_6) +aw_4w_6 + bw_4w_6
\right)x^2 \\ && + bw_1w_4w_6. 
\end{eqnarray*}

Again, since $G$ has a Hamilton path, $G$ is SR. Using the software package, Maple (\cite{maple}) to verify that these polynomials are relatively prime using the Euclidean Algorithm, it can be shown that there are no common roots between $m_w(G,x)$ and $m_w(G\setminus \{v\},x)$. Hence, it follows that $G$ is ${\rm SRSI}(v)$ for all weightings as in the figure above. Let $2, 3, 4$ be the neighbours of $v$ in $G$. In the graph $H=G \setminus \{v\}$, it is straightforward to note that $H$ is not ${ \rm SRSI}$ with respect to $2$ nor $4$ for any weighting since zero is a root of the corresponding matching polynomials of both $m_w(H\setminus \{2\},x)$
and $m_w(H\setminus \{4\},x)$, and zero is a root of $m_w(H,x)$. Regarding vertex $3$, we consider a particular weighting for the graph, $H\setminus \{3\}$; namely, weight the edge in the component, $H_1$, consisting of $K_2$ with weight equal to 1, and weight the edges of the component, $H_2=P_4$, with weights equal to 2, 1, and 2 from left to right, respectively. Then, it follows that $m_w(H_1,x)=x^2-1$ and $m_w(H_2,x)= x^4-5x^2+4$. It is easily verified that 1 is a root of both $m_w(H_1,x)$  and $m_w(H_2,x)$, so 1 is a double root of $m_w(H\setminus \{3\},x)$. Hence, $H\setminus \{3\}$ is not a SRSI$(u)$ graph with respect to all possible weightings.

The next result is a consequence of the work in the previous section, but also demonstrates that all SRSI$_w(v)$ graphs can be obtained from the graph expansion operation that was discussed in the previous subsection. Thus, it follows that all such SRSI$_w(v)$ graphs can be obtained from graph expansion by starting with a single edge which may be treated as a so-called ``seed" for constructing all connected SRSI$_w(v)$ graphs.

\begin{cor}
    Suppose $G$ is a connected graph that has a perfect matching or a nearly perfect matching, respectively. Then, there exists a vertex $v$ in $G$ such that $G = (G\setminus \{v\}) \underset{u}{\Cup}\{v\}$, where $u$ is a neighbour of $v$ in $G$, $w$ is an edge weighting of $G \setminus \{v\}$,  and $G \setminus \{v\}$ SRSI$_w(u)$. 
\end{cor}

\begin{proof}
    First, assume that $G$ has a perfect matching $M$. Choose a  vertex $v$ in $G$ that is not a cut vertex of $G$. Assume the edge $\{v,u\}$ is in $M$. Set $H=G \setminus \{v\}$. Then, $H$ is connected graph with a nearly perfect matching $M' = M \setminus \{v,u\}$, and thus $u$ is unsaturated with respect to $M'$. Applying Theorem \ref{srsi-thm}, we know that there exists a weighting $w$ of $H$ so that $H$ is SRSI$_{w}(u)$. Furthermore, it is evident that $G= H  \underset{u}{\Cup}\{v\}$. Hence, using Theorem \ref{SRSIv} it follows that $G$ is SRSI$_{w'}(v)$ for some edge weighting $w'$ on the graph $G$ with $w'|_{H}=w$. 
    
    Next, assume that $G$ has a nearly perfect matching $M$. Then by Lemma \ref{pm-npm} $G$, has a spanning tree with the same nearly perfect matching. If necessary, we may assume that $T$ has a nearly perfect matching $M'$ with a pendant vertex of $T$ being unsaturated. Thus, $T\setminus \{v\}$ and hence $H=G\setminus \{v\}$ are connected and both graphs have perfect matchings. Hence, by Theorem \ref{srsi-thm}, there exists an edge weighting $w$ of $H$ so that $H$ is SRSI$_{w}(u)$ where $u$ is any vertex adjacent to $v$ in $T$ and hence $G$. As above, we have $G= H  \underset{u}{\Cup}\{v\}$. Thus, using Theorem \ref{SRSIv}, $G$ is SRSI$_{w'}(v)$ for some edge weighting $w'$ on the graph $G$ with $w'|_{H}=w$. This completes the proof.
    \end{proof}

\subsection{SRSI Trees}

As noted above, a graph $G$ is called SRSI if, for every edge weighting of $G$, $G$ has simple roots and the removal of any vertex in $G$ results in strict interlacing. Obviously, this property imposes significant restrictions on the graph itself; however, there are plenty of infinite families of SRSI graphs, including, for example, Hamiltonian graphs. In this section, we focus on trees that are SRSI. Recall that any graph which contains a Hamilton path is SRSI$(v)$ with respect to the end points of this path. We begin by analyzing paths. 

\begin{lem}
If $P_{2k+1}$ is a path on an odd number of vertices where $k \geq 1$, then $P_{2k+1}$ is not ${ \rm SRSI}$. 
\end{lem} 

\begin{proof}
Let $P_{2k+1}$ be the path on an odd number of vertices with $k \geq 1$, and suppose that $v$ is the vertex in the middle of $P_{2k+1}$ (that is, $v$ is the vertex of equal distance to the two pendant vertices). Since $P_{2k+1}$ has an odd number of vertices, $0 \in \rho(m_w(P_{2k+1}, x))$, and $P_{2k+1}$ is an ${ \rm SR}$ graph. Consider $P_{2k+1} \setminus \{v\}$. The removal of $v$ results in two paths of length $k$, so $P_{2k+1} \setminus \{v\} = P_k \cup P_k$. Hence, if the weight function, $w$, assigns the same weights to the isomorphic edges of $P_k \cup P_k$, then it follows that all roots of $P_k \cup P_k$  are double roots. Hence $P_{2k+1} \setminus \{v\}$ is not ${\rm SR}$, which means that $P_{2k+1}$ is not ${\rm SRSI}$. 
\end{proof} 

\begin{lem}
If $P_{2k}$ is a path on an even number of vertices where $k \geq 3$, then $P_{2k}$ is not ${\rm SRSI}$. 
\end{lem} 

\begin{proof}
Let $P_{2k}$ be the path on an even number of vertices where $k \geq 3$, and, assuming the vertices are labeled in increasing order, consider the removal of vertex $3$ in $P_{2k}$. Then, $P_{2k} \setminus \{3\} = K_2 \cup P_{2k-3}$. Suppose $0 \neq \lambda_i \in \rho( m_w( P_{2k-3}, x))$, for some fixed weighting, $w$, of $P_{2k-3}$, which must exist as $k \geq 3$. Since $K_2$ can realize any nonzero root of the form, $\pm \sqrt{w_{12}}$, let $w$ be the weight function that assigns the edge $w_{12} = {\lambda_i}^2$ in $K_2$. Hence, it follows that $P_{2k} \setminus \{3\}$ has the root, $\lambda_i$, with multiplicity 2, so it is not ${\rm SR}$. Thus, $P_{2k}$ is not ${ \rm SRSI}$.
\end{proof} 

The remainder of this section relies on the next important fact concerning a special inverse eigenvalue problem for entry-wise nonnegative tridiagonal matrices with zero main diagonal.

\begin{prop}(\cite[Theorem 3.2]{AG}) 
Let an arbitrary collection of numbers, $\{ \lambda_1, \lambda_2, \ldots, \lambda_n\}$, be given. In order for this collection to be the spectral data for a Jacobi matrix with nonnegative entries of the form, 

$$R = \begin{bmatrix}
0 & r_0 & 0 & \cdots & 0 & 0 & 0 \\
r_0 & 0 & r_1& \cdots & 0 & 0 & 0 \\
0 & r_1& 0 & \cdots & 0 & 0 & 0 \\
\vdots & \vdots & \vdots& \ddots & \vdots & \vdots & \vdots \\
0 & 0 & 0 & \cdots & 0 & r_{n-3}& 0 \\
0 & 0 & 0 & \cdots & r_{n-3}& 0 & r_{n-2} \\
0 & 0 & 0 & \cdots &  0 & r_{n-2} & 0\\
\end{bmatrix}, $$
 it is necessary and sufficient that the following condition is satisfied: 
\begin{itemize}
\item[(i)]{The numbers, $\{ \lambda_1, \lambda_2, \ldots, \lambda_n\}$, are real, distinct and can be ordered so that $\lambda_1 < \lambda_2 < \cdots < \lambda_n$ where $\lambda_k = -\lambda_{n-k+1}$ $(k = 1, 2, \ldots, \lfloor \frac{n}{2} \rfloor)$ and $\lambda_{\lfloor \frac{n}{2} \rfloor + 1} = 0$ if $n$ is odd.}
\end{itemize} 
\label{Jacobi} 
\end{prop} 

\vspace*{-0.25cm}
In other words, any collection of roots that satisfies the above properties can be realized by certain matrices in $S_{+}^{o}(P_n)$ associated with the path on the vertices of $n$, $P_n$. Therefore, paths play a vital role in determining whether a root is common in the spectrum of a graph and one of its vertex-deleted subgraphs, and the minimum non-induced path cover number of a graph is a particular property of interest. 

\begin{defn}
For any graph, $G$, the \emph{minimum non-induced path cover number of $G$} is the smallest number of non-induced paths it takes to partition all the vertices in $G$, and is denoted by $\Upsilon(G)$.
\label{npathcover}
\end{defn} 

\begin{thm}
Let $T = (V,E,w)$ be a tree with any weight function, $w$. If $\Upsilon (T) \geq 2$, then $T$ is not ${ \rm SRSI}$. 
\end{thm} 

\begin{proof}
Suppose $T = (V,E,w)$ is a tree with any weight function, $w$, and assume $\Upsilon (T) \geq 2$.
Observe that for the star on $n \geq 3$ vertices, $S_n$, the removal of the middle vertex, $v_n$, results in $S_n \setminus \{v_n\} = \underset{n-1}{\underbrace{K_1 \cup K_1 \cup \cdots \cup K_1}}$, so $0$ is a root of $m_w(S_n,x)$ and $m_w(S_n \setminus \{v_n\}, x)$  by interlacing. Hence, the star is never ${ \rm SRSI}$, and we will only consider trees that are not stars. We use induction on $\Upsilon (T)$ to show that there are no ${\rm SRSI}$ trees with $\Upsilon (T) \geq 2$. Let $T \neq S_n$ be a tree with $\Upsilon (T) =2$. Consider the following two cases. 

\noindent
\textbf{Case 1:} Suppose $T$ has the construction in Figure \ref{PathC1}.

\begin{figure}[!htb]
\centering
\includegraphics[scale=.72]{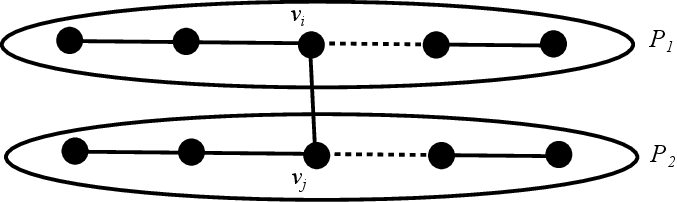} 
\caption{A tree with $\Upsilon(T) = 2.$} 
\label{PathC1}
\end{figure} 

\noindent
Then, there exists a vertex, $v_i$, such that $T \setminus \{v_i\} = P_{t_1} \cup P_{t_2} \cup P_{j}$ where $t_1 + t_2 + j = n-1$, $j \in \{1, 2\}$, and at least two of these paths have at least one edge. Since the vertex-deleted subgraph of $T$ is the disjoint union of paths, the weighted adjacency matrix  $A \in S_+^o(T\setminus \{v_i\})$ has the form 
$$A = \begin{bmatrix}
R_1 & 0 & 0 \\
0 & R_2& 0 \\
0 & 0 & R_3
\end{bmatrix},$$ 
where $R_i$ is a weighted Jacobi matrix in $S_{+}^{o}(P_{t_i})$ for each $P_{t_i}$ as presented in Proposition \ref{Jacobi}. Hence, Proposition \ref{Jacobi} and the Cauchy Interlacing Inequalities guarantee the existence of a root, $\lambda_i$, which is a common root of the characteristic polynomial of $A$, and the characteristic polynomial of a matrix in $S_+^o(T)$ obtained by bordering the matrix, $A$. Using Theorem \ref{mw-tree} we know that the characteristic polynomial of any matrix in $S_+^o(T)$ is equal to a weighted matching polynomial of $T$, and it follows that $T$ is not ${ \rm SRSI}$. 

\noindent
\textbf{Case 2:} Suppose $T$ has the construction in Figure \ref{PathC1}. 

\begin{figure}[!htb]
\centering
\includegraphics[scale=.72]{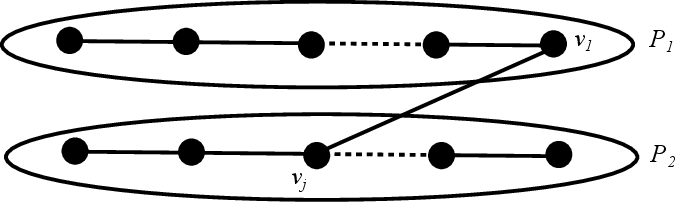} 
\caption{Another tree with $\Upsilon(T) = 2.$} 
\label{PathC2}
\end{figure} 

\noindent 
Removing vertex $v_1$ from $T$ results in $T \setminus \{v_1\} = P_{t_1} \cup P_{t_2}$ where $t_1 + t_2 = n-1$, and each $P_{t_i}$ has at least one edge. Since the vertex-deleted subgraph of $T$ is the disjoint union of at least two nonempty paths, the weighted adjacency matrix  $A \in S_+^o(T \setminus \{v_1\})$ has the form 
$$A = \begin{bmatrix}
R_1 & 0 \\
0 & R_2 
\end{bmatrix},$$ 
where $R_i$ is the weighted 
weighted Jacobi matrix in $S_{+}^{o}(P_{t_i})$ for each $P_{t_i}$ as presented in Proposition \ref{Jacobi}. Hence, Proposition \ref{Jacobi} and the Cauchy Interlacing Inequalities guarantee the existence of a root, $\lambda_i$, which is a common root of the characteristic polynomial of $A$, and the characteristic polynomial of a matrix in $S_+^o(T)$ obtained by bordering the matrix, $A$. Again, by Theorem \ref{mw-tree}  the characteristic polynomial of any matrix in $S_+^o(T)$ is equal to a weighted matching polynomial of $T$, so $T$ is not ${ \rm SRSI}$. 

Now, assume that any tree, $T \neq S_n$ with $\Upsilon (T) = k \geq 3$ is not ${\rm SRSI}$, and consider a tree $\widehat{T} \neq S_n$ with $\Upsilon (\widehat{T}) = k+1$ as described in Figure \ref{PathK}.

\begin{figure}[!htb]
\centering
\includegraphics[scale=.72]{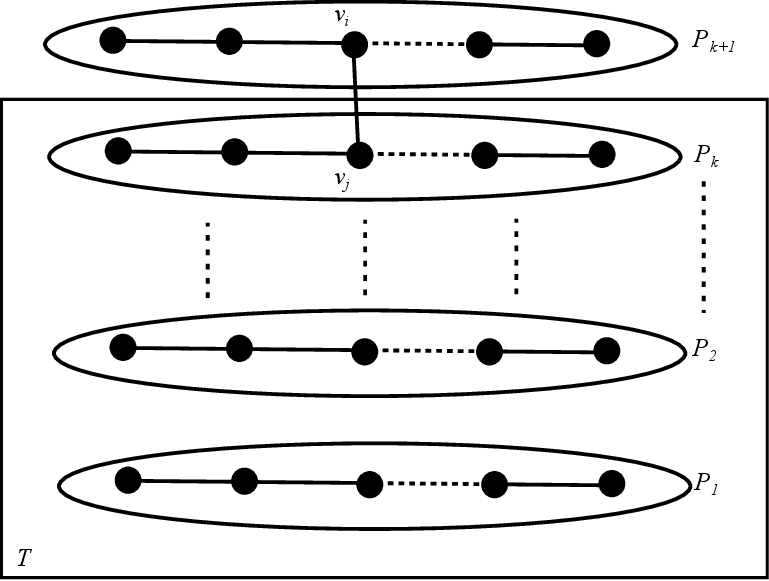} 
\caption{A tree with $\Upsilon(T) = k+1$.} 
\label{PathK}
\end{figure} 

Observe that there is only one edge, $e = \{ v_i, v_j\}$, from the induced subgraph, $T$, with $\Upsilon (T) = k$ to the subgraph, $P_{k+1}$, since $\widehat{T}$ is a tree.  If the vertex, $v_{i}$, resides in the subgraph, $P_{k+1}$, then either the removal of $v_i$ results in the subgraph, $T$, and at most two other paths in which at least one path has an edge \big(i.e., $\widehat{T} \setminus \{ v_i\} = T \cup P_{t_1}$ or $\widehat{T} \setminus \{ v_i\} = T \cup P_{t_1} \cup P_{t_2}$ \big) or the path $P_{k+1}$ is an edge or is a path on three vertices with $v_i$ in the middle. In the latter two cases, we can apply an argument similar to the one below by removing the vertex $v_j$ as pictured in Figure \ref{PathK} as $k\geq 3$. Continuing with the former case, since the vertex-deleted subgraph of $\widehat{T}$ has at least one disjoint path, Proposition \ref{Jacobi} ensures that for some fixed root, $\lambda_i$, of $m_w(T, x)$ can be realized in the spectrum of this nonempty disjoint path. Hence, $\widehat{T}$ is not ${ \rm SRSI}$. By induction, any tree with $\Upsilon (T) \geq 2$ is not ${\rm SRSI}$. 
\end{proof}

\begin{ex}
It is clear that a path on two vertices, or $K_2$, is SRSI. Consider $P_4 = (V,E, w)$ a path on four vertices with any fixed weight function, $w$. 

\begin{figure}[h]
\centering
\includegraphics[scale=.75]{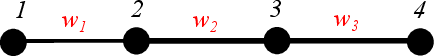} 
\caption{Determining ${\rm SRSI}_w(v)$ vertices in $P_4$.} 
\end{figure} 

Recall the path is ${\rm SR}_w$ for any weight function $w$, and the weighted matching polynomial for $P_4$ is $m_w(P_4,x) = x^4 -(w_1+w_2 +w_3)x^2 +w_1w_3$. It follows from Corollary \ref{HamPath} that $P_4$ is  ${\rm SRSI}_w(1)$ and ${\rm SRSI}_w(4)$. Without loss of generality, consider the weighted matching polynomial for $P_4 \setminus \{2\}$:
$$ m_w(P_4 \setminus \{2\},x) = x(x^2-w_3).$$
The collection of roots for this subgraph is $\rho(m_w(P_4 \setminus \{2\},x)) = \{0, \pm \sqrt{w_3} \}$, which are distinct. For each root, $ \lambda_i \in \rho(m_w(P_4 \setminus \{2\},x))$, consider the following equations of $m_w(P_4, \lambda_i)$:
\begin{align*}
m_w(P_4,0) &= w_1w_3 \neq 0, \\
m_w(P_4,\pm \sqrt{w_3}) &= w_3^2 - (w_1 + w_2 + w_3)w_3 + w_1w_3 = -w_2w_3\neq 0.
\end{align*}
Since each weight $w_i$, is positive, the roots $\{ 0, \pm \sqrt{w_3}\}$ are not roots of $m_w(P_4,x)$. Thus, $P_4$ is ${\rm SRSI}_w(2)$, and, by a similar argument using symmetry, $P_4$ is ${\rm SRSI}_w(3)$. Thus, it follows that $P_4$ is ${\rm SRSI}$ at every vertex.
\label{P4ex}
\end{ex} 

\begin{cor}
Among all trees on at least two vertices, only the paths $P_2$ and $P_4$ are SRSI.
\end{cor}

\vspace*{-0.6cm}
\section*{Acknowledgments}
S.M.\ Fallat was supported in part by an NSERC Discovery Research Grant, Application No.: RGPIN--2025-05272. J.\ Parenteau was supported during her MSc. studies in part by an NSERC CGSM scholarship and the University of Regina. Currently, J.\ Parenteau is partially supported by an NSERC PGSD scholarship and the University of Regina.

\end{document}